\documentclass[11pt]{amsart}

\usepackage{geometry}
\newtheorem{propo}{Proposition}[section]

\newtheorem{lemma}[propo]{Lemma}

\newtheorem{theo}[propo]{Theorem}

\theoremstyle{definition}

\theoremstyle{remark}
\newtheorem{remar}[propo]{Remark}

\newcommand{\Irr}{\mathop{\rm Irr}\nolimits}

\newcommand{\ZZ}{\mathop{\mathbb Z}\nolimits}

\newcommand{\al}{\alpha}

\newcommand{\de}{\delta}
\newcommand{\ep}{\varepsilon}
\newcommand{\lam}{\lambda }

\newcommand{\om}{\omega }

\newcommand{\si}{\sigma }

\begin{document}

\title[Weight zero in irreducible  representations]{Weight zero in tensor-decomposable irreducible  representations of simple algebraic groups}

\author[A. Baranov]{Alexander Baranov}
\address{Department of Mathematics, University of Leicester, Leicester, LE1 7RH, UK}
\email{baranov@le.ac.uk, alex.baranov.uk@gmail.com}
\thanks{Supported by University of Leicester}

\author[A. Zalesski]{Alexandre Zalesski}

\address{Department of Physics, Mathematics and Informatics, National Academy of Sciences of Belarus,
66 Prospekt  Nezavisimosti, Minsk, Belarus}
\email{alexandre.zalesski@gmail.com}



\dedicatory{Dedicated to A.V. Yakovlev on occasion of his 80\textsuperscript{th} birthday}

\begin{abstract}
Let $G$ be a simple algebraic group in defining characteristic $p>0$, and let $V$ be an irreducible $G$-module
which is the tensor product of exactly two non-trivial modules.
 We obtain a criterion for $V$ to have the zero weight. In addition, we provide a uniform criterion for an
 irreducible representation of a simple Lie algebra over the complex numbers  to have a multiple of a prescribed fundamental weight.
\end{abstract}

\maketitle

\section{Introduction}

The structure of the weight system is one of the major characteristics of a representation  of a simple algebraic group.
If the ground field is of characteristic $p=0$ then this structure is well understood and the general theory provides efficient tools for solving various specific questions. For $p>0$ a key result is due to Premet \cite{Pr}, on coincidence of the weight systems of irreducible representations with $p$-restricted highest weight and those with the same highest weight for $p=0$ (with some exceptions treated in \cite{z09}). In general, due to Steinberg's tensor product theorem, the weight system of an
 irreducible representation with non-restricted highest weight 
 is the sum of those for the tensor-indecomposable factors. However, this fact is insufficient to decide when a given weight
 occurs in a given representation. This paper is addressed to the following problem.

\medskip
\noindent
{\bf  Problem.}
For a simple algebraic group $G$ in defining
characteristic $p > 0$ determine the $p$-modular irreducible representations that have weight $0$.

\medskip

In this paper we solve the problem in the most natural special case where the representation  in question is a tensor product of exactly two tensor-indecomposable non-trivial factors.  In full generality the problem looks untractable. In fact, we see no way to approach it even in the case of three tensor factors.





To state the main results of the paper we introduce some notation. For a simple,
simply connected  algebraic group $G$ of type $X_n$ ($X=A,\dots,G$)
we denote by $\om_1^X,\ldots , \om_n^X$ (or simply, by $\om_1,\ldots , \om_n$) the fundamental weights of the corresponding weight system. A weight of $G$ is called dominant if
it is an integral linear combination of $\om_1,\ldots , \om_n$ with non-negative coefficients, and $p$-restricted if the coefficients do not exceed $p-1$.
We denote by $\Omega_0^+(G)$ the set of all dominant weights
$\sum_{i=1}^n a_i\om_i$ with $a_n=0$.
For a dominant weight $\lam$ we denote by $V_\lam$ the irreducible  $G$-module with highest weight $\lam$ and by $\lam^*$ the highest weight of the module  dual to $V_\lam$.
 A radical weight is one which is an integral linear combination of
the roots. For the notion of a miniscule weight see \cite[Ch. VIII, \S 7.3]{Bo8}.
For weights $\om,\si$ we write $\om\succeq \si$ if $\om-\si$ is an integral linear combination of
the roots with non-negative coefficients.  In addition, $(\om,\si)$ is the standard inner product 
on the weight lattice
of $G$, see \cite{Bo}, which is the main source of our terminology for weight systems.
For
a dominant weight  $\lam=\sum_{i=1}^n a_i\om_i^B$  of $B_n$ we denote by
$\lam^C=\sum_{i=1}^n a_i\om_i^C$ the corresponding dominant weight of $C_n$.

Let $\lam$ be the highest weight  of a rational irreducible  $G$-module $V$. It follows from Steinberg's tensor product theorem (see below) that $V$ is tensor-indecomposable if and only if  $\lam=p^i\lam_0$ for an integer $i\geq 0$ and a $p$-restricted weight $\lam_0$ (with a few exceptions for $p=2,3$); in non-exceptional cases $V$ is a tensor product of two tensor-indecomposable non-trivial modules if and only if  $\lam=p^k(\lam_0+p^i\lam_1)$, where $\lam_0,\lam_1$ are non-zero $p$-restricted weights and integers $k\geq 0, i\geq 1$.
In the exceptional cases $i=0$ is allowed. 
Our main results are the following theorems.

\begin{theo}\label{th7} Let $G$ be a simple algebraic group of rank $n$ in defining characteristic
$ p>0$.
If $G\in\{ B_n,C_n,G_2\}$, assume $p>2$.
Let $V$ be an irreducible  $G$-module of highest weight $\om$.
Suppose $V$ is the tensor product of two non-trivial tensor-indecomposable $G$-modules,
so
$\om=p^k\lam$ with $\lam=\lam_0+p^i\lam_1$, where
$\lam_0,\lam_1$ are $p$-restricted dominant weights and $k,i\geq 0$.
Then $V$ has weight $0$ if and only if
$\lam $ is radical and either $\lam_1$ is radical or
$\lam_1\succeq \om_j$ for some miniscule $\om_j$ and
$\lam^*_0\succeq p^i\om_j$ (equivalently,
$(\lam^*_0,\om_j)\geq p^i(\om_j,\om_j)$).\end{theo}

Theorem \ref{th7} solves the problem for the generic case. Our next theorem settles the exceptional case of  $G\in\{ B_n,C_n,G_2\}$ for $p=2$.

\begin{theo}\label{th8} In  notation of Theorem {\rm \ref{th7}} suppose that $p=2$  and $G\in\{ B_n,C_n,G_2\}$. Then V has weight $0$ if and only if $\lam$ is radical and one of the following holds.
\begin{itemize}
\item[$(1)$]
  $G=C_n$ and either \\ $(1.1)$   $i\geq 0$, $\lam_0\in \Omega_0^+(G)$, $\lam_1=\om_n$
and $\lam_0\succeq 2^i\om_n$
(equivalently, $(\lam_0,\om_n)\geq 2^in$), or  \smallskip \\ \smallskip
$(1.2)$     $i\ge1$, $\lam_0,\lam_1\in \Omega_0^+(G)$ and $\lam_1$ is radical; or \\ \smallskip
$(1.3)$  $i\ge1$, $\lam_0,\lam_1\in \Omega_0^+(G)$,
$\lam_1\succeq \om_1$
and $\lam_0\succeq 2^i\om_1$ (equivalently, $(\lam_0,\om_1)\geq 2^i$).

\medskip
\item[$(2)$] 
$G=B_n$ and  either \\  \smallskip
$(2.1)$   $i\ge1$, $\lam_0\in \Omega_0^+(G)$, $\lam_1=\om_n$,
$\lam_0^C+2^{i-1}\om_n^C$ is a radical weight of $C_n$,
$\lam_0^C\succeq 2^{i-1}\om_n^C$; or \\
$(2.2)$  $i\ge1$, $\lam_0,\lam_1\in \Omega_0^+(G)$,
 $\lam_0^C+2^{i}\lam_1^C$ is a radical weight of $C_n$ satisfying $(1.2)$ or $(1.3)$.
\medskip

\item[$(3)$] 
$G=G_2$ and  either \\
$(3.1)$ $i\ge1$ and $\om_1\notin\{\lam_0, \lam_1\}$; or \\ \smallskip
$(3.2)$ $i=1$, $\lam_0=\om_1+\om_2$, $\lam_1=\om_1$.
\end{itemize}
\end{theo}

Conditions for a weight $\sum a_i\om_i$ to be radical can be easily read off the expressions of the fundamental weights in terms of the simple roots given in \cite[Planche I - Planche IX]{Bo} (and elsewhere), however, for reader's convenience we
provide them  at the end of this paper in Table \ref{tr}.
The condition for $\lam^*_0$ in Theorem \ref{th7} is explicitly given in terms of the $a_i$'s in Table  \ref{tc}.
A general criterion for $\om\succeq m\om_j$ (used in  Theorems \ref{th7} and \ref{th8})  is given in the following  Theorem \ref{th1},
which is also of independent interest.

\begin{theo}\label{th1} Let $\Omega$ be the weight lattice of  a simple Lie algebra $L$ over the complex numbers,
and let $\om_1,\ldots , \om_n$ be the fundamental weights of $\Omega$. Let $m> 0$ and $ 1\leq j\leq n$
be  integers and  let $\om\in \Omega$.

$(1)$ $\om\succeq m\om_j$ if and only if  $\om- m\om_j$ is radical and
$(\om,\om_j)\geq m(\om_j,\om_j)$.

$(2)$ Suppose that $\om$ is dominant. Let $\rho$ be an irreducible representation  of $L$ with highest weight  $\om$. Then $m\om_j$ is a weight of $\rho$ if and only if   $\om- m\om_j$ is radical and $(\om,\om_j)\geq m(\om_j,\om_j)$.
\end{theo}

Note that the condition $(\om,\om_j)\geq m(\om_j,\om_j)$
above can be written in a more explicit form in terms of the coefficients of the expression $\om=\sum_i a_i\om_i$, see Remark \ref{rw}.
In particular, Theorem \ref{th1} gives a simple criterion
for the weight system of an irreducible representation of a simple Lie algebra
to contain a multiple of a given fundamental weight, in terms of a single inequality for the coefficients of the highest weight.
The proof of  Theorem \ref{th1} is based on the following interesting geometric property of the fundamental weights, which is also equivalent to
non-negativity of certain $2\times 2$ minors of the inverse Cartan matrix of a simple Lie algebra
(see Proposition \ref{an5}).

\begin{propo}\label{an6} For~ all ~$1\leq i,j,k\leq n$ we have
$(\om_i,\om_j)(\om_k,\om_k)\geq (\om_i,\om_k)(\om_k,\om_j)$.
\end{propo}

The last two results can be regarded as a contribution to the weight system theory.

\medskip
{\it Notation} Let $G$ be a simple, simply connected algebraic group of rank $n$. To say that
 $G$ is of type $A_n$ etc., we simply write $G=A_n$.  The Weyl group of $G$ is denoted by $W.$ We denote by $\al_1,\ldots , \al_n$ and $\om_1,\ldots , \om_n$, respectively, the simple roots and the fundamental weights of the weight lattice $\Omega(G)$ of $G$. The zero weight (or weight 0) is often denoted by 0. We write $\Omega^+(G)$ for the set of dominant weights, which are exactly the non-negative integral linear combinations of $\om_1,\ldots , \om_n$ (ordered as in \cite{Bo}).
We denote by $\Omega_0^+(G)$ the set of all dominant weights
$\sum_{i=1}^n a_i\om_i$ with $a_n=0$.
For a dominant weight  $\lam=\sum_{i=1}^n a_i\om_i^B$  of $B_n$ we denote by
$\lam^C=\sum_{i=1}^n a_i\om_i^C$ the corresponding dominant weight of $C_n$.

 The expressions $\om_i=\sum d_{ij}\al_j$ of $\om_i$ in terms of $\al_j$
for $1\leq i,j\leq n$ are given in
\cite[Planche I - Planche IX]{Bo}; the corresponding matrix $\Delta=(d_{ij})$ is the inverse of the Cartan matrix of $G$.

There is a standard partial ordering of elements  of $\Omega$; specifically, for $\mu,\mu'\in\Omega$ we write
$\mu\preceq\mu'$ and $\mu'\succeq \mu$ if and only if  $\mu'-\mu$ 
is a sum of simple roots.
Every irreducible  $G$-module has a unique weight $\om$ such that $\mu\prec\om$ for every weight $\mu$ of $V$ with $\mu\neq \om$.
This is called the  highest weight of $V$. There is a bijection between $\Omega^+$ and
the set $\Irr (G)$ of rational irreducible $G$-modules,
so for $\om\in\Omega^+$ we denote by $V_\om$ the irreducible  $G$-module with highest weight $\om$.
A weight $\om=\sum a_i\om_i$ is called $p$-restricted if $0\leq a_1,\ldots , a_n<p$, and we call a module $V_\om$ $p$-restricted if $\om$ is $p$-restricted.

\section{Preliminaries}

Every dominant weight $\om=\sum a_i\om_i$ has a unique $p$-adic expansion $\om=\lam_0+p\lam_1+\cdots +p^k\lam _k$ for some $k$, where $\lam_0,\ldots , \lam_k$ are $p$-restricted dominant weights. This yields the tensor product decomposition $V_\om=V_{\lam_0}\otimes V_{p\lam_1}\otimes\cdots \otimes V_{p^k\lam_k}$, which is known as Steinberg's  tensor product theorem
(see \cite[Theorem 41]{St}). We set $e(G)=1$ for $G\in\{A_n,E_6,E_7,E_8\}$, $e(G)=2$ for $G\in\{B_n,C_n, F_4\}$ and $e(G)=3$ for $G=G_2$. The following
result describes $p$-restricted tensor-indecomposable modules (note that (3) is a special case of (2)).

\begin{lemma}\label{pp1} \cite[Corollary of Theorem 41]{St},\cite[1.6]{Se}  Let $G$ be a simple algebraic group and let $V$ be a $p$-restricted irreducible  $G$-module with
 highest weight $\om$.
\begin{itemize}
\item[$(1)$] If $p\ne e(G)$ then $V$ is tensor-indecomposable.
\item[$(2)$] Suppose $p=e(G)$.
Let $\om=\sum a_i\om_i$  and let $\om=\lam_1+\lam_2$,
where $\lam_1=\sum_{i:\, \al_i\, \text{\rm is short}} a_i\om_i$ and
$\lam_2=\sum_{i:\, \al_i\,\text{\rm is long}} a_i\om_i$. Then
$V=V_{\lam_1}\otimes V_{\lam_2}$, and $V_{\lam_k}$ is tensor-indecomposable for $k=1,2$.
In particular, $V$ is    tensor-indecomposable if and only if  $\lam_1=0$ or $\lam_2=0$.
\item[$(3)$] Suppose $p=2$ and $G=B_n,C_n$. Then $V$ is tensor-indecomposable
if and only if  $\om=\om_n$ or $\om\in \Omega_0^+(G)$.
\end{itemize}
\end{lemma}

\begin{remar}\label{rfr}
If $\om$ is not $p$-restricted then
$V_\om$ is tensor-decomposable unless the expansion $\om=\lam_0+p\lam_1+\cdots +p^k\lam _k$
consists of a single term, say, $p^i\lam_i$. In that case,
$V_{p^i\lam_i}\cong V_{\lam_i}^{\text{Fr}^i}$ is $i$th Frobenius twist of
$V_{\lam_i}$, and thus it is tensor-indecomposable if and only if   $V_{\lam_i}$ is so.
\end{remar}

The following is well known.

\begin{lemma}\label{fr} Let $\om\in\Omega^+(G)$ and let $i\in \mathbb{N}$. Then
the system of weights of the $G$-module $V_{p^i\om}$ is exactly
$\{p^i\lam\mid \lam \text{ weight of } V_{\om}\}$. In particular,
$0$ is a weight of $V_{\om}$ if and only if  $0$ is a weight of $V_{p^i\om}$.
\end{lemma}

Let $V$ be a tensor-indecomposable irreducible $G$-module
with $p$-restricted highest weight $\om$.
We wish to find conditions for the module $V$ to have weight 0.
A well known obvious necessary condition is that $\om$ is a radical weight.
Premet's Theorem \cite{Pr} shows that this condition is also sufficient if
$p=0$ or $p>e(G)$.

\begin{theo}\label{prem} \cite[Theorem 1]{Pr}  Assume $p=0$ or $p>e(G)$. Let $V_\om$ be an irreducible $G$-module
with $p$-restricted dominant weight $\om$. Then the set of dominant weights in
 $V_\om$ is $\{\mu\preceq \om \mid \mu\in\Omega^+(G) \}$, and the set of  weights in
 $V_\om$ is $\{w\mu \mid  \mu\in \Omega^+(G),\ \mu\preceq\om,\ w\in W\}.  $ Consequently, if $\om$ is radical then $V_\om$
 has weight $0$, otherwise $V_\om$ has a unique miniscule fundamental weight $\om_j$. \end{theo}

The additional claim is stated explicitly (with a proof) in \cite[Proposition 2.3]{SZ06}.
We wish to characterize the occurrence of weight 0 in the exceptional cases as well.
This was essentially done in \cite{z09},
except the case of $B_n$ for $p=2$. The latter can be treated similarly to the case of
$C_n$ as the corresponding root systems are dual to each other.
Indeed, by \cite[Theorem 28]{St} (see also \cite[Section 5.4]{Hum}), for $p=2$
there are homomorphisms ({\em special isogenies}) of algebraic groups
$\varphi_{BC}:B_n\to C_n$
and
$\varphi_{CB}:C_n\to B_n$
which are isomorphisms of the abstract groups (as the ground field is algebraically closed), but not isomorphisms of algebraic groups  \cite[p.155 Remark(b)]{St}.
More exactly, suppose we have
a pair of simple, simply connected algebraic groups $G$ and
$G'$ with dual root systems $R$ and $R^\vee$. Let $p=2,3$ and let
$\varphi:G\to G'$ be the associated special isogeny. When $\varphi$
is restricted to a map of maximal tori $T\to T'$, the comorphism
$\varphi^*$ sends each short root of $G'$ to the corresponding
long root of $G$ and each long root of $G'$ to the $p$th multiple of
the corresponding short root of $G$.
The corresponding fundamental weights are mapped similarly:
$\varphi^*(\om_i^{G'})=\om_i^{G}$ if the root $\alpha_i$ of $G'$ is short
and $\varphi^*(\om_i^{G'})=p\om_i^{G}$ if the root $\alpha_i$ of $G'$ is long.
In addition, the composition of $\varphi$ with an irreducible representation $V$ of
$G'$ is an irreducible representation of $G$, denoted $\varphi^*(V)$.
Note that the module $V$ has weight zero if and only if
$\varphi^*(V)$ has it,
as the zero weight space corresponds to a trivial representation
of a maximal torus of the group.
Thus, we have the following theorem.

\begin{theo}\label{BC}\cite[Theorem 28]{St}, \cite[Section 5.4]{Hum}
Let $\varphi:G\to G'$, $\varphi_{BC}:B_n\to C_n$ and
$\varphi_{CB}:C_n\to B_n$ be special isogenies as above.
\begin{itemize}
\item[$(1)$] If $V$ is an irreducible $G'$-module of highest weight $\lambda$
then $\varphi^*(V)$  is an irreducible $G$-module
of highest weight $\varphi^*(\lambda)$.
Moreover, $V$ has weight zero if and only if
$\varphi^*(V)$ has it.
\item[$(2)$] $\varphi_{CB}^*(\sum_{i=1}^n a_i\om_i^B)=\sum_{i=1}^{n-1}2a_i\om_i^C+a_n\om_n^C$
for all integers $a_1,\dots,a_n$.
\item[$(3)$] $\varphi_{BC}^*(\sum_{i=1}^n a_i\om_i^C)=\sum_{i=1}^{n-1}a_i\om_i^B+2a_n\om_n^B$
for all integers $a_1,\dots,a_n$.
\end{itemize}
\end{theo}

Now we can
characterize the occurrence of weight 0 in the exceptional cases as well.

\begin{propo}\label{zz9}  Let $G$ be a simple algebraic group
and let $V$ be a tensor-indecomposable
irreducible $G$-module with  $p$-restricted highest weight $\om$.
Then
$V$ has weight $0$ if and only if $\om$ is radical and none of the following holds:
\begin{itemize}
\item[$(1)$] $p = 2$, $G=C_n$, $n$ even and $\om=\om_n$;
\item[$(2)$] $p = 2$, $G=B_n$, $\om\in \Omega_0^+(G)$ and $\om^C$ is not a radical weight of $C_n;$
\item[$(3)$] $p = 2$, $G=G_2$ and $\om=\om_1.$
\end{itemize}
\end{propo}

\begin{proof} This was proved in \cite[Corollary 11]{z09} except for the case $G=B_n$ for $p = 2$.
Note that in \cite[Corollary 11]{z09} the case with $p = 2$ and $G=C_n$, $n$ even, $\om=\om_n$  is missing, but it is clear from the context or from Corollary 14 in \cite{z09} that this has to be included.

Suppose now that $G=B_n$ and $p = 2$. Since  $V_\om$
is tensor-indecomposable, by Lemma \ref{pp1}, $\om=\om_n$ or $\om\in \Omega_0^+(G)$.
Note that $\om_n$ is not radical for $B_n$ (see Table \ref{tr}), so
$V_{\om_n}$ cannot have weight 0. Thus, we can assume
$\om\in \Omega_0^+(G)$.
Hence $\om=\varphi_{BC}^*(\om^C)$ by Theorem \ref{BC}(3).
By Theorem \ref{BC}(1), $V_\om$ has  weight $0$ if and only if  $V_{\om^C}$ has it.
Since $\om^C\ne\om_n^C$, this is equivalent to $\om^C$ being radical.
 \end{proof}

We will also need the following fact.

\begin{lemma}\label{zz9b}
 \cite[Proposition 13]{z09}
 Let $p = 2$, $G=C_n$ and let $\om$ be a $2$-restricted
dominant weight of $G$ with $\om\in\Omega_0^+(G)$.
Let $\mu$ be a dominant weight of $G$ such that $\mu\preceq\om$. Then $\mu$ is a weight of
$V_\om$.
 \end{lemma}

We complement this with the following  observation.

\begin{propo}\label{z27}
Let $G$ be a simple algebraic group
and let $V$ be a tensor-indecomposable
irreducible $G$-module with  highest weight $\om$.
Then $V$ has a unique minimal dominant weight $\om_{\min}$.
Moreover, if $\om$ is $p$-restricted then
$\om_{\min}$ is the same as in characteristic zero (i.e.
$\om_{\min}$ is $0$ if  $\om$ is radical, or miniscule otherwise)
unless one of the following holds.
\begin{itemize}
\item[$(1)$] $p=2$, $G=C_n$ and $\om=\om_n$. Then $\om_{\min}=\om_n$.
\item[$(2)$] $p=2$, $G=B_n$ and $\om\in \Omega_0^+(G)$.
Then $\om_{\min}=0$ if $\om^C$ is a radical weight of $C_n$ and  $\om_{\min}=\om_1$  otherwise.
\item[$(3)$]  $p=2$, $G=G_2$ and $\om=\om_1$. Then $\om_{\min}=\om_1$.
\end{itemize}
\end{propo}

\begin{proof} Let $V$ be a tensor-indecomposable $G$-module and let
$\om$ be the highest weight of $V$. Since $V$ is tensor-indecomposable, by Steinberg's tensor product theorem, $\om=p^k\om'$ for some $p$-restricted $\om'$.
Now if $V_{\om'}$ has the smallest dominant weight $\om_{\min}$ then
it follows from Lemma  \ref{fr} that $p^k\om_{\min}$ is the smallest dominant weight of $V_\om$ and we are done.
Thus, we can assume that $\om$ is $p$-restricted.
In view of Theorem \ref{prem}, we need only to consider the exceptional cases.

First, let $p=2$ and $G=C_n$. Then by Lemma \ref{pp1}(3),
$\om=\om_n$ or $\om\in\Omega_0^+(G)$.
If $\om=\om_n$ then $\om_n$ is well known to be the unique dominant weight in $V_\om$, so
$\om_{\min}=\om_n$.
If  $\om\in\Omega_0^+(G)$ then
by Lemma \ref{zz9b},
the weights of $V_\om$
are the same as the weights of the corresponding module over the complex numbers, so the result
follows.

Let  $p=2$ and $G=B_n$. Then, by Lemma \ref{pp1}(3),
$\om=\om_n$ or $\om\in\Omega_0^+(G)$.
If $\om=\om_n$ then $V_\om$ is the spin module and $\om_n$ is its unique dominant weight,
so $\om_{\min}=\om_n$ (as in characteristic 0).
Suppose  $\om\in\Omega_0^+(G)$.
Then, by Theorem \ref{BC}, $V_\om=\varphi_{BC}^*(V_{\om^C})$.
By the above, $V_{\om^C}$ has the smallest dominant weight $\om_{\min}'$.
Note that $\om_{\min}'=0$ if $\om^C$ is radical, and  $\om_{\min}'=\om_1^C$  otherwise.
We claim that $\varphi_{BC}^*(\om_{\min}')$ is the smallest weight of
$V_\om$.
Indeed, this is clear if $\om^C$ is radical as $\varphi_{BC}^*(0)=0$.
Suppose $\om^C$ is not radical. Then $V_\om$ has no zero weight but has weight
$\varphi_{BC}^*(\om_1^C)=\om_1$. It is well known that $\om_1$ is minimal.
It remains to show uniqueness.
Let $\mu$  be  any dominant weight of $V_\om$. Then $\mu\ne0$. We wish to show that $\om_1\preceq\mu$.
Since  $\om\in\Omega_0^+(G)$, both $\om$ and $\mu$ are radical weights of $B_n$, see  Table \ref{tr}.
Hence $\mu$ is a combination of $\om_1,\dots,\om_{n-1}$ and $2\om_n$ with non-negative integral coefficients.
Since $\om_1\preceq\om_j$ for $j=1,\dots,n-1$ and $\om_1\preceq 2\om_n$
(see \cite[Planche II]{Bo}), we get that $\om_1\preceq\mu$, as required.

If $G=F_4$ or $G=G_2$  then every weight of  $G$ is radical, so, by Proposition \ref{zz9}, every $G$-module has weight 0, unless  $G=G_2$ and $p=2$. In the latter case $V_{\om_1}$ is the only irreducible  $G$-module that does not have weight 0, and in this case $\om_1$ is the only dominant weight of  $V_{\om_1}$.
So the result follows. \end{proof}

\section{Some observations on weight lattices}
\label{ss}

The aim of this section is to prove Theorem \ref{th1}, i.e.
to establish a criterion for an irreducible representation with highest weight $\om$
of a finite dimensional simple complex Lie algebra (or simple complex algebraic group) $G$
to contain a multiple of a fundamental weight $m\om_j$.
A number of useful facts on the weight lattices can be found in \cite{Bo} and \cite{Ste}, but the following  simple lemma
was not recorded there, and is probably not 
recorded elsewhere.  Note that $(\om_i,\al_j)=\de_{ij}(\al_j,\al_j)/2$ (the Kronecker delta), see \cite[Ch.VI, \S 1.10 ]{Bo}.

\begin{lemma}\label{r1} Let $\mu,\nu\in\Omega(G)$. Then $\mu\succeq \nu$ if and only if  $\mu-\nu$ is radical and $(\mu,\om_k)\geq (\nu,\om_k)$ for $k=1,\ldots , n$.
\end{lemma}

\begin{proof} The ``only if'' part. Let  $\mu\succeq \nu$. By definition of $\succ$, we have   $\mu= \nu +\sum b_j\al_j$, where $b_j\geq 0$ are integers. So
 $(\mu,\om_k)= (\nu,\om_k)+\sum b_j (\al_j,\om_k)$, whence $(\mu,\om_k)\geq (\nu,\om_k)$ as $(\al_j,\om_k)\geq 0$ for all $j,k\in\{1,\ldots , n\}$.

Conversely,  $\mu= \nu +\sum b_j\al_j$ with $b_j\in\ZZ$ as  $\mu-\nu$ is radical. So  $(\mu,\om_k)\geq (\nu,\om_k)$
means $(\mu,\om_k)-(\nu,\om_k)\geq 0$, whence  $\sum b_j(\al_j,\om_k)\geq 0$ for all $k$. As $(\al_j,\om_k)=0$ for $k\neq j$
and $(\al_k,\om_k)>0$ for all $k$, we conclude that $b_j\geq 0$ for all $j$.
\end{proof}

Set $E(G)=(e_{ij})$, where $e_{ij}=(\om_i,\om_j)$ for $1\leq i,j\leq n$.
Let $\om=\sum a_i\om_i$.
The expressions $\om_i=\sum _k d_{ik}\al_k$ of the fundamental weights in terms of the simple roots are available in \cite[Planche I -- Planche  IX]{Bo}. 
By substituting these into $\om=\sum a_i\om_i$ we get
 $\om=\sum_k c_k\al_k$ where
 $c_k=\sum_i d_{ik}a_i$.
As $(\om_j,\al_k)=\frac{(\al_k,\al_k)}{2}\de_{jk}$  \cite[Ch.VI, \S 1.10 ]{Bo},
we have $$(\om,\om_k)=\sum_i c_i (\al_i,\om_k)=\sum_i c_i (\om_k,\al_i)=\frac{1}{2}(\al_k,\al_k)c_k,$$ whence $c_k=\frac{2}{(\al_k,\al_k)}(\om,\om_k)$. As $\om_j=\sum _k d_{jk}\al_k$, we have
\begin{equation}\label{fd}
d_{jk}=\frac{2}{(\al_k,\al_k)}(\om_j,\om_k)=\frac{2}{(\al_k,\al_k)}e_{jk}
\quad \text{ for }\quad 1\leq j,k\leq n.
\end{equation}
Inspection of the root systems implies that $\frac{1}{2}(\al_k,\al_k)=1$ (and hence $d_{jk}=e_{jk}$) unless one of the following  holds:

$G=B_n$, $k=n$ and $\frac{1}{2}(\al_n,\al_n)=1/2;$

$G=C_n$, $k=n$ and $\frac{1}{2}(\al_n,\al_n)=2;$

$G=F_4$, $k=3,4$ and $\frac{1}{2}(\al_k,\al_k)=1/2;$

$G=G_2$, $k=2$ and $\frac{1}{2}(\al_2,\al_2)=3$.

\noindent From these and the formulas for $\om_i=\sum _k d_{ik}\al_k$ in \cite{Bo} one can easily write down
the matrix $E(G)$.

The matrix $E(G)$ is symmetric and satisfies the following  condition
which is very useful in what follows.

\begin{propo}\label{an5} We have
\begin{equation}\label{1om}
(\om_i,\om_j)(\om_k,\om_k)\geq (\om_i,\om_k)(\om_k,\om_j)~ for~ all ~1\leq i,j,k\leq n.
\end{equation}
\noindent In other words, the matrix $E(G)$ satisfies the property
 $e_{ij}e_{kk}\geq e_{ik}e_{kj}$ for all integers  $i,j,k$ with $1\leq i,j,k\leq n$.
 Moreover, the matrix $\Delta(G)=(d_{ij})$ satisfies the property
 $d_{ij}d_{kk}\geq d_{ik}d_{kj}$ for all $1\leq i,j,k\leq n$.\end{propo}

\begin{proof} Note that the result for the matrix $\Delta(G)$ follows from that for the
matrix $E(G)$ as
the condition $e_{ij}e_{kk}\geq e_{ik}e_{kj}$ remains valid under replacing $E(G)$ by $XE(G)Y$, where $X$ and $Y$ are arbitrary
diagonal matrices, and in fact these conditions are equivalent if $X$ and
$Y$ are non-degenerate.

If $k=i$ or $k=j$ then we have the equality in (\ref{1om}). Moreover, the case with $i=j$ where (\ref{1om}) takes shape  $(\om_i,\om_i)(\om_k,\om_k)\geq (\om_i,\om_k)^2$ follows from the Cauchy-Schwarz inequality. So we assume that $i,j,k$ are distinct. In addition, (\ref{1om}) is symmetric with respect to $i,j$, so we can assume $i<j$.

Let $G=A_n$.
Then $e_{ij}=i(n+1-j)/(n+1)$ for $1\leq i\leq j\leq n$. So we have
$$(n+1)^2e_{ij}e_{kk}= i(n+1-j)k(n+1-k) \text{ for } i< j$$
\noindent and
$$(n+1)^2e_{ik}e_{kj}=
\begin{cases}
k(n+1-i)\cdot k(n+1-j)& \text{ if } k<i< j, \cr
i(n+1-k)\cdot  k(n+1-j) & \text{ if } i< k< j, \cr
i(n+1-k)\cdot j(n+1-k) & \text{ if } i< j< k.
\end{cases}$$
\noindent
Clearly, (\ref{1om}) holds in all these cases.

\medskip

For the remaining classical types the entries $e_{ij}=(\om_i,\om_j)$ can be easily computed  if one embeds
the weight lattice into an orthogonal space with orthonormal basis $\ep_1,\ldots , \ep_n$, defined in
\cite[Planche II-Planche IX]{Bo}, where the expressions
of $\om_1,\ldots , \om_n$ in terms of $\ep_1,\ldots , \ep_n$ are written down.
In addition,  (\ref{1om}) remains unchanged if $\om_i$ is replaced by $\om_i'=q_i\om_i$ for any $q_i> 0$.

\medskip
Let $G=B_n$. Then
$\om_i=\ep_1+...+\ep_i$ for $i<n$ and $\om_n=(\ep_1+...+\ep_n)/2$.
Put $\om_i'=\om_i$ for $i<n$ and $\om_n'=2\om_n$. Then
$$
(\om_i',\om_j')=(\ep_1+...+\ep_i,\ep_1+...+\ep_j)=\min\{i,j\}
\quad \text{ for all } 1\le i,j\le n. \\
$$
Hence, as required, for all $i<j$ and for all $k$ we get
$$
(\om_i',\om_j')(\om_k',\om_k')=ik\ge \min\{i,k\}\min\{k,j\}=(\om_i',\om_k')(\om_k',\om_j').
$$

\medskip
The case $G=C_n$ follows from that with $G=B_n$. 
As the matrices $(d_{ij})$ computed for $G=B_n$ and $G=C_n$ are transposes of each other,
the claim follows.

\medskip
Let $G=D_n$. 
Put $\om_i'=\om_i$ for $i<n-1$ and $\om_i'=2\om_i$ for $i=n-1,n$. Then
$$
(\om_i',\om_j')=(\ep_1+...+\ep_i,\ep_1+...+\ep_j)=\min\{i,j\}
\quad \text{ for all } 1\le i,j\le n \\
$$
except $(\om_{n-1}',\om_{n-1}')=n$ and $(\om_{n-1}',\om_{n}')=n-2$.
Therefore, as in the case $B_n$,  (\ref{1om}) holds for all $i,j,k$ such that
$n-1\not\in \{i,j,k\}$ or $n\not\in \{i,j,k\}$.
Suppose $\{n-1,n\}\subseteq \{i,j,k\}$. Then, as required, we have
\begin{align*}
(\om_i',\om_j')(\om_k',\om_k')&=in> i(n-2)=(\om_i',\om_k')(\om_k',\om_j') \quad \text{ if } i\leq n-2<j,k; \\
(\om_i',\om_j')(\om_k',\om_k')&=(n-2)k> kk=(\om_i',\om_k')(\om_k',\om_j') \quad \text{ if } k\leq n-2<i<j.
\end{align*}
For the exceptional groups the result follows by inspection of the matrix $E(G)$.
\end{proof}

Note that the matrix $\Delta(G)$ is the inverse of  the Cartan matrix. Both of them are known to satisfy many combinatorial properties. However, we have been unable to find in the literature anything related to Proposition \ref{an5}.

The following  statement is a special case of Lemma \ref{r1}.

\begin{lemma}\label{cm9} Let $m\geq 0$ be an integer and let $\om\in\Omega(G)$.
Then $\om\succeq m\om_j$ if and only if  $\om-m\om_j$ is a radical weight and
 for every $k=1,\ldots , n$ we have $(\om,\om_k)\geq m(\om_j,\om_k)$.
\end{lemma}

\begin{lemma}\label{r6} Let $m\geq 0$ be an integer and let $\om\in\Omega(G)$. Then  $(\om,\om_k)\geq m(\om_k,\om_k)$ implies $(\om,\om_i)\geq m(\om_k,\om_i)$ for $i=1,\ldots , n$.
\end{lemma}

\begin{proof}
Let $\om=\sum a_i\om_i=\sum_k c_k\al_k$.
Recall that $c_k=\sum_l a_ld_{lk}$,
$c_k=\frac{2}{(\al_k,\al_k)}(\om,\om_k)$ and   $d_{ij}=\frac{2}{(\al_j,\al_j)}(\om_i,\om_j)$.
Hence the conditions $(\om,\om_k)\geq m(\om_k,\om_k)$
and
$(\om,\om_i)\geq m(\om_k,\om_i)$ in the lemma are equivalent to
$c_k\geq md_{kk}$ and
$c_i\geq md_{ki}$, respectively.
Therefore, we need to show that $c_k\geq md_{kk}$ implies $c_i\geq md_{ki}$ for $i=1,\ldots , n$.
Note that $d_{ij}>0$ for every $i,j$. Then  we get
$$
c_i=\sum_l a_ld_{li}\geq \sum_l a_l(d_{lk}d_{ki}/d_{kk})=
c_k d_{ki}/d_{kk}\geq
(md_{kk}) d_{ki}/d_{kk}=
md_{ki},
$$as required.
\end{proof}

\begin{proof}[Proof of Theorem {\rm\ref{th1}}]
(1). By Lemma \ref{cm9}, $\om\succeq m\om_j$ if and only if  $\om- m\om_j$ is radical and $(\om,\om_i)\geq m(\om_j,\om_i)$ for all $i\in\{1,\ldots , n\}$.  By Lemma \ref{r6},
$(\om,\om_j)\geq m(\om_j,\om_j)$ implies $(\om,\om_i)\geq m(\om_j,\om_i)$ for all $i=1,\ldots , n$,
whence the result. (2) follows from (1) and \cite[Ch. VIII, \S 7.2, Proposition 5(iv)]{Bo8}
(or see  Theorem \ref{prem} for $p=0$). \end{proof}

\begin{remar}\label{rw}
By using the exact values of $(\om_i,\om_j)$ for all $i,j$,
we can rewrite the condition $(\om,\om_j)\geq m(\om_j,\om_j)$
in Theorem \ref{th1} in a more explicit form. For example, let $G=A_n$
and let $\om=\sum b_k\om_k$ be a dominant weight of $G$ such that $\om-m\om_j$ is radical.
We wish to rewrite the condition $(\om,\om_j)\geq m(\om_j,\om_j)$
in terms of $b_k$'s. By using the values of
$(\om_k,\om_j)$ (see the proof of Proposition \ref{an5}), we get:
$$
(n+1)(\om,\om_j)=\sum_{k=1}^{n} b_k(n+1)(\om_k,\om_j)=
\sum_{k=1}^{j} b_kk(n+1-j)+\sum_{k=j+1}^{n} b_kj(n+1-k).
$$
%
%
%
%
Thus,  $(\om,\om_j)\geq m(\om_j,\om_j)$ is equivalent to
$$
\sum_{k=1}^{j} b_kk(n+1-j)+\sum_{k=j+1}^{n} b_kj(n+1-k)\geq mj(n+1-j).
$$
By dividing both parts by $j(n+1-j)$, we get the condition
\begin{equation}\label{m3}
\sum_{k=1}^{j}\frac{k}{j} b_k+\sum_{k=j+1}^{n}\frac{n+1-k}{n+1-j} b_k\geq m.
\end{equation}
All other types are considered similarly.
\end{remar}

\section{Occurrence of weight zero}

Let $V_\lam$ be an irreducible  $G$-module with highest weight $\lam$. Then $\lam=\sum_{i=0}^k p^i\lam_i$, where $k\geq 0$ and $\lam_i$
is a $p$-restricted dominant weight for every $i$. By Steinberg's theorem, $V_\lam=\otimes V_{p^i\lam_i}$.
By Lemma \ref{fr}, we can assume $\lam_0\neq 0$ as $V_\lam$ has weight 0 if and only if  $V_{\lam/p^j}$ has, where $j$ is the minimal $i$ with $\lam_i\neq 0$.
Set $\lam'=\sum_{i>0} p^i\lam_i$, so $V_\lam=V_{\lam_0}\otimes V_{\lam'}$. Obviously, $V_\lam$ has weight 0 if and only if
$V_{\lam_0}$ has a weight $\nu$ such that $-\nu$ is a weight of $V_{\lam'}$. So the problem reduces to that of the existence
of such $\nu$. A well known necessary (but not sufficient)  condition is that $\lam$ is a radical weight.
Recall (see the introduction) that we consider only the case where
$V_{\lam}$ is a tensor product of exactly two tensor-indecomposable irreducible  modules. By Lemma \ref{pp1},
with a few well known exceptions, $V_{\lam_0}$ is tensor-indecomposable, and then
 $V_\lam=V_{\lam_0}\otimes V_{p^i\lam_i}$ for some fixed $i\ge 1$. The exceptions arise only for $p=2$ and $G\in \{B_n,C_n, F_4, G_2\}$ and for $G_2$ for $p=3$, that is, for $p\leq e(G)$, see Theorem \ref{prem}.
The exceptional cases  require a special consideration and will be dealt with in the next section. In this section we consider the regular case of $p> e(G)$.

\begin{lemma}\label{pr1} Let $V$ be an irreducible  $G$-module of highest weight  $\lam=\lam_0+p^i\lam_1$, where
$\lam_0,\lam_1$ are p-restricted dominant weights, $i\ge1$ and  $p> e(G)$. Let
 $\mu=p^i\lam_1$
and  let $\nu$  be the  minimal dominant weight of $V_\mu$.
Then the zero weight occurs in $V$ if and only if  $\nu^*$ is a weight of $V_{\lam_0}$ (equivalently, $\lam_0^*\succeq\nu$).
\end{lemma}

\begin{proof} The ``if'' part. If $\lam_0^*\succeq\nu$ then $V_{\lam_0}$ has weight $-\nu$
\cite[Ch.VIII, \S.7.1, Proposition 11]{Bo8}
and hence $V_{\lam_0}\otimes V_\mu$
has weight 0. Then $V$ has weight 0.

The ``only if'' part. Clearly, if $V$ has weight 0 then there is weight $\si\in V_\mu$ such that $-\si\in V_{\lam_0}$. Using the Weyl group, we can assume  $\si$ dominant. By Proposition \ref{z27}, $V_\mu$ has a unique minimal dominant weight, $\nu$, say. Then $\nu\preceq\si$. On the other hand,  $\si$ is a dominant weight in $V_{\lam_0^*}$.
This implies  $\nu\preceq \si\preceq \lam_0^*$.
Since $p> e(G)$, by Theorem \ref{prem},  $\nu\preceq \lam_0^*$ implies that $\nu $ is a weight of $V_{\lam_0^*}$, as claimed.
\end{proof}

\begin{theo}\label{th6} Let $G$ be a simple algebraic group of rank $n$ and let $V$ be an irreducible  $G$-module of highest weight $\lam=\lam_0+p^i\lam_1$, where
$\lam_0,\lam_1$ are $p$-restricted dominant weights, $i\ge 1$ and  $p> e(G)$.
Then $V$ has weight $0$ if and only if  the weight $\lam $ is radical and one of the following  holds:
\begin{itemize}
\item[$(1)$]    $\lam_1$ is radical;
\item[$(2)$] $\lam_1$ is not radical and
$\lam_0^*\succeq  p^i\om_j$ (equivalently,
$(\lam^*_0,\om_j)\geq p^i(\om_j,\om_j)$), where $\om_j$ is a miniscule weight such that $\lam_1\succeq \om_j$. 
\end{itemize}
\end{theo}

 \begin{proof} If $\lam_1$ is radical then so is $p^i\lam_1;$ hence $V$ has weight 0 if and only if  $\lam_0$ is radical (see Theorem \ref{prem}).

 Suppose that $\lam_1$ is not radical. By Theorem \ref{prem}, $\lam_1\succeq\om_j$ for some miniscule weight $\om_j$,
 where $j\in\{1,\ldots , n\}$, and $\om_j$ is a dominant weight of $V_{\lam_1}$, obviously  minimal.
 By Lemma \ref{pr1}, $V$ has weight 0 if and only if
 $p^i\om_j$ is a weight of $V_{\lam^*_0}$.
 By Theorem \ref{th1}, this holds if and only if  $(\lam^*_0,\om_j)\geq p^i(\om_j,\om_j)$.
 \end{proof}

We wish to express the condition $(\lam^*_0,\om_j)\geq p^i(\om_j,\om_j)$ in a more
explicit form, as at the end of Section \ref{ss} (with $m=p^i$).
Let $\lam_0=\sum_k a_k\om_k$.
Note that in $G=A_n$ every fundamental weight is miniscule, whereas in $E_8,F_4,G_2$ all weights are radical,
so no miniscule weight exists.
For the other groups these are $\om_n$ in $B_n,$ $\om_1$ in $C_n$, $\om_1,\om_{n-1},\om_n$ in $D_n,$ $\om_1,\om_6$ in $E_6$, $\om_7$ in $E_7$. Note that $\lam_0^*=\lam_0$ if $G$ is of type $B_n,C_n$ or $E_7$.

Let $G=A_n$.
Let $\om=\lam_0^*=\sum_k b_k\om_k$ then $b_k=a_{n+1-k}$ for all $k$, so by swapping $k$ and $n+1-k$ in (\ref{m3}) and substituting $a_k$'s we rewrite the condition
$(\lam^*_0,\om_j)\geq p^i(\om_j,\om_j)$ as
\begin{equation}\label{m4}
\sum_{k=1}^{n-j}\frac{k}{n+1-j}a_{k} +\sum_{k=n+1-j}^{n}\frac{n+1-k}{j}a_{k}\geq p^i.
\end{equation}

Let $G=B_n$.
Then
$(\lam_0,\om_n)=\sum _k a_k(\om_k,\om_n)=\sum_{k=1}^{n-1} a_k\frac{k}{2}+a_n\frac{n}{4} $,
see the proof of Proposition \ref{an5}. So
$(\lam_0,\om_n)\geq p^i(\om_n,\om_n)$ can be written as $\sum_{k=1}^{n-1}ka_k+na_n/2\geq p^in/2 $.

The remaining cases are considered similarly. We summarize our calculations in Table \ref{tc} below.
We write $\lambda_0^*=(b_1,\dots,b_n)$ if
$\lambda_0^*=\sum b_k\om_k$. Recall that in the case $G=D_n$ we have $\lambda_0^*=(a_1,\dots,a_{n-2},a_{n-1}^*,a_n^*)$
where $a_k^*=a_k$ if $n$ is even (i.e. $\lambda_0^*=\lambda_0$);
$a_{n-1}^*=a_n$ and $a_{n}^*=a_{n-1}$  if $n$ is odd.

\noindent
\begin{table}[h]
\centering
\caption{\label{tc}
Condition $(\lam^*_0,\om_j)\geq p^i(\om_j,\om_j)$ for non-radical
$\lam_0=(a_1,\dots,a_n)$ and miniscule $\om_j$.}
\renewcommand{\arraystretch}{1.4}
\begin{tabular}{c|c|c|c}
  \hline
 Type & $\lambda_0^*$ & Miniscule $\om_j$ & $(\lam^*_0,\om_j)\geq p^i(\om_j,\om_j)$ \\\hline \hline
 $A_n$ & 
               $(a_n,a_{n-1},\dots,a_1)$ & $\om_j\ (1\le j\le n)$  &
       $\sum\limits_{k=1}^{n-j}\frac{k}{n+1-j}a_{k} +\sum\limits_{k=n+1-j}^{n}\frac{n+1-k}{j}a_{k}\geq p^i$
                  \\ \hline
 $B_n$ & $\lambda_0^*=\lambda_0$ & $\om_n$  & $\sum_{k=1}^{n-1}ka_k+na_n/2\geq p^in/2 $ \\ \hline
 $C_n$ & $\lambda_0^*=\lambda_0$ & $\om_1$  & $\sum_{k=1}^{n}a_k\geq p^i$  \\ \hline
       &   & $\om_1$ & $\sum_{k=1}^{n-2}a_k+a_{n-1}/2+a_{n}/2\geq p^i$ \\ \cline{3-4}
 $D_n$ & 
         $(a_1,\dots,a_{n-2},a_{n-1}^*,a_n^*)$  & $\om_{n-1}$ & $\sum_{k=1}^{n-2}ka_k+na_{n-1}^*/2 + (n-2)a_n^*/2\geq p^in/2 $ \\ \cline{3-4}
       &  & $\om_n$ & $\sum_{k=1}^{n-2}ka_k+(n-2)a_{n-1}^*/2 + na_n^*/2\geq p^in/2 $ \\ \hline
 $E_6$ & $(a_6,a_2,a_5,a_4,a_3,a_1)$ & $\om_1$ & $2a_1+3a_2+4a_3+6a_4+5a_5+4a_6\geq 4p^i$ \\ \cline{3-4}
       &                               & $\om_6$ & $4a_1+3a_2+5a_3+6a_4+4a_5+2a_6\geq 4p^i$ \\ \hline
 $E_7$ & $\lambda_0^*=\lambda_0$ & $\om_7$ & $2a_1+3a_2+4a_3+6a_4+5a_5+4a_6+3a_7\geq 3p^i$ \\ \hline
\end{tabular}

\end{table}
\begin{remar}
As $a_i<p$, we observe that in Table \ref{tc} the left hand side in the inequality for $E_6$ is at most $24(p-1)$. Thus, if $p^{i}>6p-6$
(i.e. either $i\ge3$ or $i=2$ and $p>3$)
then $V_\lam$ has weight 0 if and only if  both $\lam_0,\lam_1$ are radical.
Similarly, for $E_7$, the left hand side of the inequality
is at most $27(p-1), $ so $p^{i}>9(p-1)$ implies that $V_\lam$ has weight 0 if and only if  both $\lam_0,\lam_1$ are radical.
\end{remar}

\section{Exceptional cases}

In this section we consider the cases with $p\le  e(G)$, i.e.
$G=B_n,C_n, F_4$ for $p=2$ and
$G=G_2$ for $p=2,3$.

Suppose first that $G=B_n,C_n$ and $p=2$.
Recall that $\Omega^+_0$ is the set of $p$-restricted dominant weights $\mu=\sum_{i=1}^n b_i\om_i$ of $G$ such that $b_n=0$.
Let $\mu$ be a $p$-restricted dominant weight of $G$. Then by Lemma \ref{pp1}, $V_{\mu}$
is tensor-indecomposable if and only if   $\mu=\om_n$ or  $\mu\in \Omega^+_0$. Thus, Lemma \ref{pp1} and Remark \ref{rfr} imply the following.

\begin{lemma}\label{ti}
Let $G=B_n,C_n,$ $p=2$, $\lam\in\Omega^+(G)$ and $V=V_{\lam}$.
Then $V$ is the tensor product of exactly two non-trivial tensor-indecomposable irreducible  $G$-modules
if and only if
$\lam=2^k(\lam_0+2^i\lam_1)$ with non-zero $p$-restricted $\lam_0$ and $\lam_1$,
$k,i\ge0$ and one of the following holds:
\begin{itemize}
\item[$(1)$] $i=0$, \ $\lam_0\in \Omega^+_0$ and $\lam_1=\om_n$;
\item[$(2)$] $i\ge1$, \ $\lam_0\in \Omega^+_0$ and $\lam_1=\om_n$;
\item[$(3)$] $i\ge1$, \ $\lam_0\in \Omega^+_0$ and $\lam_1\in \Omega^+_0$;
\item[$(4)$] $i\ge1$, \ $\lam_0=\om_n$ and $\lam_1=\om_n$;
\item[$(5)$] $i\ge1$, \ $\lam_0=\om_n$ and $\lam_1\in \Omega^+_0$.
\end{itemize}
\end{lemma}

\begin{theo}\label{c22} Let $G=C_n,$ $p=2$, and let $V=V_{\lam}$
where $\lam=\lam_0+2^i\lam_1$ is  as in Lemma {\rm \ref{ti}} with $k=0$.
Then $V_\lam$ has weight $0$ if and only if  $\lam$ is radical and one of the following  holds:
\begin{itemize}
\item[(a)]
 $i\geq 0$, $\lam_0\in \Omega_0^+(G)$, $\lam_1=\om_n$
and $\lam_0\succeq 2^i\om_n$,
equivalently, $(\lam_0,\om_n)\geq 2^in$;
\item[(b)]
 $i\ge1$, $\lam_0,\lam_1\in \Omega_0^+(G)$ and $\lam_1$ is radical;
\item[(c)]
 $i\ge1$, $\lam_0,\lam_1\in \Omega_0^+(G)$,
$\lam_1\succeq \om_1$
and $\lam_0\succeq 2^i\om_1$, equivalently, $(\lam_0,\om_1)\geq 2^i$.
\end{itemize}
\end{theo}

\begin{proof} We need to consider Cases (1)-(5) in Lemma \ref{ti}.
Case (1) was settled in \cite[Theorem 2.7]{Z19}, but this can be handled uniformly with Case (2).

Consider Cases (1) and (2). We have $\lam_0\in \Omega^+_0$,
$\lam=\lam_0+2^i\om_n$  and $V=V_{\lam_0}\otimes V_{2^i\om_n}$ with $i\ge0$.
Recall that $\om_n$ is the only dominant weight of $V_{\om_n}$
and the weights of $V_{2^i\om_n}$ form a single orbit under the Weyl group of $G$.
Hence $V$ has weight 0 if and only if  $2^i\om_n$ is a weight of $V_{\lam_0}$.
By Lemma \ref{zz9b}, this holds if and only if  $\lam_0\succeq 2^i\om_n$ (as $\lam_0\neq \om_n$).
By Theorem \ref{th1}, this is equivalent to
the inequality $(\lam_0,\om_n)\geq 2^i(\om_n,\om_n)$. As $(\om_n,\om_n)=n$,
we get (a).

Consider Case (3). Then
it follows from Lemma \ref{zz9b} that $V_{\lam_0}$, respectively, $ V_{2^i\lam_1}$ has weight 0 if and only if  $\lam_0$, respectively, $ \lam_1$, is radical. If  $ \lam_1$ is radical then $V_{2^i\lam_1}$ has weight 0. In this case $V$ has weight 0 if and only if  so has $V_{\lam_0}$,
which is equivalent to $\lam_0$ to be radical by Lemma \ref{zz9b}. Thus, we get (b).

Suppose that $\lam_1$ is not radical. Then $\lam_1\succ \om_1$.  As in the proof of Lemma \ref{pr1}, let $\si$ be a common weight of $V_{\lam_0}$ and $ V_{2^i\lam_1}$, which can be assumed to be dominant. Then  $ 2^i\om_1\preceq\si$. As $\si\preceq \lam_0$,
we conclude by Lemma \ref{zz9b} that $2^i\om_1$ is a weight of $V_{\lam_0}$.
 As above, this is equivalent to
the inequality $(\lam_0,\om_1)\geq 2^i(\om_1,\om_1)=2^i$ by Theorem \ref{th1}, and we get (c).

In Cases (4) and (5), $V_\lam$ cannot have weight $0$ as otherwise  $V_{\om_n}$ would contain
the weight  $2^i\si$ for some dominant weight $\si\preceq \lam_1$.
\end{proof}

\begin{theo}\label{b22} Let $G=B_n,$ $p=2$, and let $V=V_{\lam}$
where $\lam=\lam_0+2^i\lam_1$ is  as in Lemma {\rm \ref{ti}} with $k=0$.
Then $V_\lam$ has weight $0$ if and only if  one of the following  holds:
\begin{itemize}
\item[(a)]  $i\ge1$, $\lam_0\in \Omega_0^+(G)$, $\lam_1=\om_n$,
$\lam_0^C+2^{i-1}\om_n^C$ is a radical weight of $C_n$,
$\lam_0^C\succeq 2^{i-1}\om_n^C$;
\item[(b)]
 $i\ge1$, $\lam_0,\lam_1\in \Omega_0^+(G)$,
 $\lam_0^C+2^{i}\lam_1^C$ is a radical weight of $C_n$ satisfying
 the conditions  of Theorem {\rm \ref{c22}(b)(c)}.
\end{itemize}
\end{theo}

\begin{proof} We will consider Cases (1)-(5) in Lemma \ref{ti}.
By Theorem \ref{BC}, $V_\lam$ has weight $0$ if and only if  the $C_n$-module
 $U=\varphi^*_{CB}(V_\lam)$ has   weight $0$.
Moreover, $U$ is an irreducible $C_n$-module with highest weight $\varphi^*_{CB}(\lam)$.
Recall that
 $\varphi_{CB}^*(\sum_{i=1}^{n}a_i\om_i^B)=\sum_{i=1}^{n-1}2a_i\om_i^C+a_n\om_n^C$.

In Cases (1), (4) and (5) we have, respectively, \\
\begin{tabular}{l}
$\quad$ $i=0$ and $\varphi^*_{CB}(\lam)=\om_n^C+2\lam_0^C$ with $\lam_0\in \Omega^+_0$, \\
$\quad$ $i\ge1$ and $\varphi^*_{CB}(\lam)=\om_n^C+2^i\om_n^C$, \\
$\quad$ $i\ge1$ and $\varphi^*_{CB}(\lam)=\om_n^C+2^{i+1}\lam_1^C$ with $\lam_1\in \Omega^+_0$.
\end{tabular}
\\
By Theorem \ref{c22}, $U$ cannot have weight 0 in all these cases.

Consider Case (2). Then  $i\ge1$ and $\varphi^*_{CB}(\lam)=2\lam_0^C+2^i\om_n^C$
with $\lam_0\in \Omega^+_0$. Let $\mu=\lam_0^C+2^{i-1}\om_n^C$.
Then $U$ has weight 0 if and only if  the $C_n$-module
 $V_\mu$ has   weight $0$. The corresponding condition is given in
 Theorem \ref{c22}(a), so we get the part (a) of the theorem.

Consider Case (3). Then  $i\ge1$ and $\varphi^*_{CB}(\lam)=\lam^C=\lam_0^C+2^i\lam_1^C$
with $\lam_0,\lam_1\in \Omega^+_0$.
Then $U$ has weight 0 if and only if  $\lam^C$ satisfies the conditions of Theorem \ref{c22}(b)(c),
so we get the part (b) of the theorem.
\end{proof}

Next we consider the exceptional cases for $G=F_4,G_2$.

\begin{lemma}\label{ex3} If  $G=F_4$ for $p=2$ or $G_2$ with $p=3$ then every irreducible
$G$-module has weight $0$.
\end{lemma}

\begin{proof} By Proposition \ref{zz9}, in these cases every irreducible representation  of $G$  has weight 0, so the result follows.
\end{proof}

We are left with $G_2$ with $p=2$. Then the 2-restricted weights are $0,\om_1,\om_2$ and $\om_1+\om_2$,
and if $\lam$ is 2-restricted then  $V_{\lam}$  has weight 0 if and only if  $\lam\neq\om_1$.
Note that the subdominant weights of
$\om_1+\om_2$ are $2\om_1,\om_2,\om_1,0$, and those of $\om_2$ are $\om_1,0$.

\begin{lemma}\label{g22a} Let $\lam=\lam_0+2^i\lam_1$, where $i\geq 1$ and $\lam_0,\lam_1\neq 0$ are $2$-restricted dominant weights.
Then $V_{\lam}$ has weight $0$ if and only if  $(1)$ $\om_1\notin\{\lam_0, \lam_1\}$ or $(2)$ $i=1$, $\lam_0=\om_1+\om_2$, $\lam_1=\om_1$.
\end{lemma}

\begin{proof} Note that $V_{\om_1}$ has no weight $\si=2a_1\om_1+2a_2\om_2$ with $a_1,a_2\in\ZZ$. Suppose  that $i>1$ or $i=1$
and $\lam_0=\om_1$. Then
 $V_{\lam}$ has weight $0$ if and only if
each $V_{\lam_0}$ and $V_{\lam_1}$ has weight 0. This happens if and only if  (1) holds.

Let $i=1$ and $\lam_0\neq\om_1$.  Then $V_{\lam_0}$ has weight 0. As $V_{\lam_1}$ has weight 0 if and only if   $\lam_1\neq\om_1$, we are left with the case $\lam_1=\om_1$. Using the Weyl group we observe that  $V_{\lam_0}$ must have weight $2\om_1$. By inspection, we conclude that this happens if and only if  $\lam_0=\om_1+\om_2$. \end{proof}

\begin{proof}[Proof of Theorems $\ref{th7}$ and $\ref{th8}$] These follow from combining together Theorems \ref{th6}, \ref{c22}, \ref{b22} and Lemmas \ref{ex3}, \ref{g22a}.  \end{proof}



For reader's convenience we provide in Table \ref{tr} the well known conditions for a weight to be radical.

\begin{table}[h]
   \centering
 \caption{\label{tr} Conditions for a weight $\lam = \sum a_i \omega_i$ to be radical.}
 \renewcommand{\arraystretch}{1.4}

\begin{tabular}{|c|c|}
\hline
$A_n$ & $\sum_{i=1}^n i a_i \equiv 0 \pmod{n+1}$ \\
\hline
$B_n$ & $a_n \equiv 0 \pmod{2}$ \\
\hline
$C_n$ & $\sum_{i=1}^{ n } ia_i
 \equiv 0 \pmod{2}$ \\
 \hline
 $D_{2n}$ & $\sum_{i=0}^{n-2} a_{2i+1}\equiv  a_{2n-1} \equiv  a_{2n}  \pmod{2}$\\
 \hline
$D_{2n+1}$ & $2 (\sum_{i=0}^{n-1} a_{2i+1}) + a_{2n} - a_{2n+1} \equiv 0 \pmod{4}$ \\
\hline
$E_6$ & $a_1 - a_3 + a_5 - a_6 \equiv 0 \pmod{3}$ \\
\hline
$E_7$ & $a_2 + a_5 + a_7 \equiv 0 \pmod{2}$ \\\hline
$E_8$, $F_4$, $G_2$ & every weight is radical \\
\hline
\end{tabular}

\end{table}


 \end{document}